\documentclass[11pt]{amsart}
\usepackage{float}
\usepackage[usenames]{color}
\usepackage{graphicx}
\usepackage{amssymb}
\usepackage{amscd}
\usepackage{makecell}
\usepackage{tikz}
\usetikzlibrary{graphs}
\usetikzlibrary{arrows,decorations.markings,plotmarks,decorations.markings}
\usetikzlibrary{positioning}	
\usetikzlibrary{snakes} 

\usepackage{enumitem}

\theoremstyle{plain}
\newtheorem{thm}{Theorem}[section]
\newtheorem{lemma}[thm]{Lemma}
\newtheorem{cor}[thm]{Corollary}

\theoremstyle{definition}
\newtheorem{defi}[thm]{Definition}

\newtheorem{example}[thm]{Example}

\newcommand{\bbH}{\mathbb{H}}

\newcommand{\bbR}{\mathbb{R}}

\DeclareMathOperator{\Ad}{Ad}

\DeclareMathOperator{\spin}{spin}
\begin{document}

\title{Spin norm and lambda norm}

\author{Chao-Ping Dong}
\address[Dong]{School of Mathematical Sciences, Soochow University, Suzhou 215006,
	P.~R.~China}
\email{chaopindong@163.com}

\author{Du Chengyu}
\address[Du]{School of Mathematical Sciences, Soochow University, Suzhou 215006,
	P.~R.~China}
\email{cydu0973@suda.edu.cn}

\abstract{Given a $K$-type $\pi$, it is known that its spin norm (due to first-named author) is lower bounded by its lambda norm (due to Vogan). That is, $\|\pi\|_{\rm spin}\geq \|\pi\|_{\rm lambda}$. This note aims to describe for which $\pi$ one can actually have equality. We apply the result to tempered Dirac series. In the case of real groups, we obtain that the tempered Dirac series are divided into $\#W^1$ parts among all tempered modules with real infinitesimal characters.}

\endabstract



\maketitle


\newcommand{\td}{\mathrm{d}}
\newcommand{\C}{\mathrm{C}}
\newcommand{\e}{\mathrm{e}}
\newcommand{\id}{\mathrm{id}}
\newcommand{\coker}{\mathrm{coker}}
\newcommand{\im}{\mathrm{im~}}
\newcommand{\Hom}{\mathrm{Hom}}

\newcommand{\FF}{\mathscr{F}}
\newcommand{\GG}{\mathscr{G}}
\newcommand{\oo}{\mathcal{O}}

\newcommand{\Z}{\mathbb{Z}}
\newcommand{\proj}{\mathbb{P}}
\newcommand{\cpl}{\mathbb{C}}
\newcommand{\re}{\mathbb{R}}

\newcommand{\selfb}[1]{\noindent\fbox{\parbox{\textwidth}{#1}}}

\newcommand{\homo}{\tilde{H}}
\newcommand{\homd}{H^\Delta}
\newcommand{\homcw}{H^{\mathrm{CW}}}

\newcommand{\norm}[1]{\Vert #1 \Vert}

\newcommand{\fk}{\mathfrak{k}}
\newcommand{\fg}{\mathfrak{g}}
\newcommand{\ft}{\mathfrak{t}}
\newcommand{\fs}{\mathfrak{s}}

\newcommand{\normlambda}[1]{\Vert #1 \Vert_{\rm lambda}}
\newcommand{\normspin}[1]{\Vert #1 \Vert_{\spin}}

\section{Introduction}

Let $G$ be a \emph{linear} real reductive Lie group which is in the \emph{Harish-Chandra class} \cite{HC}. That is,
\begin{enumerate}
	\item[$\bullet$] $G$ has only a finite number of connected components;
	\item[$\bullet$] The derived group $[G,G]$ has finite center;
	\item[$\bullet$] The adjoint action $\Ad(g)$ of any $g\in G$ is an inner automorphism of $\fg=(\fg_0)_\cpl$, where $\fg_0$ is the Lie algebra of $G$.
\end{enumerate}

Let $\theta$ be a Cartan involution of $G$. We assume the subgroup $K=G^\theta$ of fixed points of $\theta$ is a maximal compact subgroup of $G$. Let $\fg_0=\fk_0\oplus\fs_0$ be the corresponding Cartan decomposition of $\fg_0$. We drop the subscript for the complexification.

Let $\hat{G}^{\rm temp,o}$ denote the set of irreducible tempered representations with real infinitesimal character (up to equivalence). Let $\hat{K}$ denote the set of $K$-types.
The following bijection was noted by Trapa \cite{T}, after Vogan's paper \cite{V84}.

\begin{thm}\label{backgroundthm}
	Let $X$ be any irreducible tempered $(\fg,K)$-module with real infinitesimal character. Then $X$ has a unique lowest $K$-type which occurs with multiplicity one. Moreover, the map
	$$
	\phi: \hat{G}^{\rm temp,o}\to\hat K
	$$
	defined by taking the lowest $K$-type, is a well-defined bijection.
\end{thm}

Motivated by the lambda norm introduced by Vogan \cite{V81}, the first-named author introduced spin norm  \cite{D13} for the classification of \emph{Dirac series} (i.e., irreducible unitary representations of $G$ with non-vanishing Dirac cohomology). We refer the reader to \cite{HP} and references therein for the notion of Dirac cohomology.  It was proven in \cite{D13} that the spin norm of a $K$-type $\pi$ is bounded below by its lambda norm. That is,
\begin{equation}\label{the-ineq}
\norm{\pi}_{\spin}\geqslant\norm{\pi}_{\rm lambda}.
\end{equation}
This inequality turns out to have a nice interpretation in the setting of Theorem \ref{backgroundthm}. Indeed, let $\hat{G}^{\rm temp,d}$ collect the members of $\hat{G}^{\rm temp, o}$ with non-zero Dirac cohomology. Put
\begin{equation*}
	\hat K_e:=\{ \pi\in \hat K |~
	\norm{\pi}_{\spin}=\norm{\pi}_{\rm lambda}
	\}.
\end{equation*}

\begin{thm} \emph{(\cite{DD})} \label{baseThm}
The map $\phi$ restricts to $\hat{G}^{\rm temp,d}$ is a bijection onto $\hat K_e$. More precisely, any member
$\pi\in\hat{G}^{\rm temp,o}$ is a Dirac series if and only if the inequality \eqref{the-ineq} becomes an equality on its unique lowest $K$-type.
\end{thm}

Given an arbitrary $K$-type $\pi$, it is not easy to compute neither $\norm{\pi}_{\rm lambda}$ nor $\norm{\pi}_{\spin}$. Thus, it is subtle to detect whether the inequality \eqref{the-ineq} is strict or not.
This note aims to give a criterion on this aspect. Our main result is Theorem \ref{MainThm}. The main idea is to insert an intermediate value between $\norm{\pi}_{\rm lambda}$ and $\norm{\pi}_{\spin}$. As an application, our result suggests that the tempered Dirac series should be separated into $\#W^1$ parts. See \eqref{def-W1} for the definition of $W^1$.

The note is outlined as follows: In Section 2, we recall lambda norm and spin norm. Then we deduce our main result in Section 3. The last section considers tempered Dirac series.

\section{Preliminaries}

In this section, we briefly recall the definitions of the spin norm and the lambda norm.

\subsection{The lambda norm}
We keep the notations $K,G,\fk,\fs,\theta$, etc as in the previous section. Let $T$ be a maximal torus of $K$ and $\ft_0$ be the Lie algebra of $T$. Recall that the analytic Weyl group is defined by
$$W(\fk,\ft)=N_K(T)/A_K(T).$$
It acts on the root system $\Delta(\fk,\ft)$. Fix a choice of positive roots $\Delta^+(\fk,\ft)$, and define
$$R(G):=\{ r\in W(\fk,\ft) | r\Delta^+(\fk,\ft)=\Delta^+(\fk,\ft) \}.$$

Given a $K$-type $\pi$, by Lemma 0.1 of \cite{SRV}, the collection of highest weights of $\pi$ as $\fk$-module is a single orbit of $R(G)$ on $\hat T\in i\ft^\ast_0$, where $\hat T$ is the abelian group of characters of $T$.

Now given any $K$-type $\pi$, take a highest weight $\mu$ of it. Then $\mu\in i\ft^\ast_0$ is dominant integral for $\Delta^+(\fk,\ft)$. Denote by $\rho_c$ the half sum of all roots in $\Delta^+(\fk,\ft)$. Choose a positive root system $\Delta^+(\fg,\ft)$ making $\mu+2\rho_c$ dominant. Denote by $\rho$ the half sum of all roots in $\Delta^+(\fg,\ft)$. Let $P$ be the projection map to the dominant chamber $C(\fg)$ corresponding to $\Delta^+(\fg,\ft)$. Then $\norm{P(\mu+2\rho_c-\rho)}$ is independent of the choices of $\mu$ and $\Delta^+(\fg,\ft)$, cf. Section 1 and Corrollary 2.4 of \cite{SRV}. Now we are ready to talk about the lambda norm.

\begin{defi} (\rm\cite{V81,Ca})
	For any $\pi\in\hat K$, the lambda norm of $\pi$ is defined to be
	\begin{equation}
		\normlambda{\pi}:=\norm{P(\mu+2\rho_c-\rho)},
	\end{equation}
	where $\mu$ is any highest weight of $\pi$. For any irreducible admissible $(\fg,K)$-module $X$, the lambda norm of $X$ is defined to be
	\begin{equation}
		\normlambda{X}:=\min_{\pi} \normlambda{\pi},
	\end{equation}
	where $\pi$ runs over all the $K$-types occurring in $X$. A $K$-type $\pi$ is called a lowest $K$-type of $X$ if it occurs in $X$ and $\normlambda{\pi}=\normlambda{X}$.
\end{defi}

\subsection{The spin norm}

Although the original definition of the spin norm involves the spin module $S_G$ of the Clifford algebra $C(\fs)$, our discussion here does not need a deep understanding of it. Our tool is mainly the root systems and their Weyl groups.

\begin{defi} (\rm\cite{D13})
	For any $\pi\in \hat K$, its spin norm is defined to be
	$$
	\normspin{\pi}:=\min \norm{\gamma+\rho_c},
	$$
	where $\gamma$ runs over all the highest weights of the $\tilde K$-types in $\pi\otimes S_G$. For any irreducible admissible $(\fg,K)$-module $X$, its spin norm is defined to be
	$$
	\normspin{X}:=\min_{\pi} \normspin{\pi},
	$$
	where $\pi$ runs over all the $K$-types occurring in $X$. We call $\pi$ a spin lowest $K$-type of $X$ if it occurs in $X$ and $\normspin{\pi}=\normspin{X}$.
\end{defi}

\section{When is the inequality \eqref{the-ineq} strict?}

 We fix a positive root system $\Delta^+(\fk,\ft)$, and denote the half sum of roots in it by $\rho_c$. Let $W(\fg,\ft)$ (resp., $W(\fk,\ft))$) be the Weyl group of $\Delta(\fg,\ft)$ (resp., $\Delta(\fk,\ft)$). Let $C(\fk)$ be the closed dominant Weyl chamber for $\Delta^+(\fk, \ft)$. For any $\mu\in\ft^*$, we use $\{\mu\}$ to denote the unique weight in $C(\fk)$ to which $\mu$ is conjugate under the action of $W(\fk, \ft)$. Let $\Delta^+(\fg, \ft)$ be a positive root system of $\Delta(\fg, \ft)$ containing $\Delta^+(\fk, \ft)$.

\begin{lemma}\emph{(\cite[Lemma 3.5]{D13})}\label{lemmaD}
	For any $K$-type $\pi$ with a highest weight $\mu\in\ft^\ast$, we have
	\begin{equation}
		\norm{\mu}_{\spin}=\min_{w\in W^1}
		\norm{\{\mu-w\rho+\rho_c\}+\rho_c},
	\end{equation}
	where
	\begin{equation}\label{def-W1}
		W^1:=\{w\in W(\fg,\ft)| wC(\fg)\subseteq C(\fk) \}.
	\end{equation}
\end{lemma}

\begin{lemma}\emph{(\cite[\S 13.3, Lemma B]{Hum})}\label{DominantNorm}
	Let $\lambda\in C(\fk)$. Then
	$$
	\norm{\lambda+\rho_c}\geqslant\norm{w\lambda+\rho_c}
	$$
for any $w\in W(\fk, \ft)$. Moreover,	the equality holds if and only if $\lambda=w\lambda$.
\end{lemma}

\begin{lemma}\label{LammaForLambdaAndRho}
	Let $\Delta$ be a root system with Weyl group $W$. Fix a positive set $\Delta^+$ of roots and denote by $\rho$ the half sum of all positive roots. For any dominant weight $\lambda$, we have the following inequality
\begin{equation}\label{LemmaOfNorms}
	\norm{\lambda-\rho}\leqslant\norm{\lambda-w\rho},~\forall w\in W.
\end{equation}
	Moreover, if $\lambda$ is dominant with respect to $w \Delta^+$, we have
	$$\norm{\lambda-\rho}=\norm{\lambda-w\rho}$$
	Otherwise, the inequality \eqref{LemmaOfNorms} is strict.
	\begin{proof}
We first prove the inequality. Compute the following difference
\begin{equation}\tag{$\ast$}
\norm{\lambda-\rho}^2-\norm{\lambda-w\rho}^2=-2(\lambda,\rho-w\rho).
\end{equation}
A widely known fact is that $\rho-w\rho$ is a sum of positive roots. The weight $\lambda$ is dominant by assumption. Thus the pairing $(\lambda,\rho-w\rho)$ is non-negative, and \eqref{LemmaOfNorms} follows.

Now suppose $\lambda$ is dominant with respect to $w \Delta^+$. Notice that the half sum of positive roots with respect to $w\Delta^+$ is $w\rho$. Applying \eqref{LemmaOfNorms} to  $\lambda$, $w\Delta^+$ and $w\rho$ gives
$$\norm{\lambda-w\rho}\leqslant\norm{\lambda-w^{-1}(w\rho)}=\norm{\lambda-\rho}.$$
Therefore, \eqref{LemmaOfNorms} becomes an equality in the current setting.

Now suppose $\lambda$ is not dominant with respect to the new positive set $w \Delta^+$. Define
$$
D_w:=\{\gamma\in\Delta^- |
 \gamma\in w\Delta^+\},
$$
where $\Delta^-=-\Delta^+$. It is well-known that
$$
\rho-w\rho=\sum_{\gamma\in D_w}(-\gamma).
$$
By assumption, $\lambda$ is not dominant with respect to $w \Delta^+$. There must exist $\beta\in w\Delta^+$ such that $(\lambda,\beta)<0$. But it cannot live in $\Delta^+$, because $\lambda$ is dominant with respect to $\Delta^+$. As a consequence, $\beta\in D_w$. Continuing with ($\ast$), we have that
$$
	-(\lambda,\rho-w\rho)=
	-\left(\lambda,\sum_{\gamma\in D_w}(-\gamma)\right)
	=\left(\lambda,\sum_{\gamma\in D_w}\gamma\right)
	\leqslant(\lambda,\beta)<0.
$$
Thus \eqref{LemmaOfNorms} is strict in this case.
	\end{proof}
\end{lemma}

  Let us state the main result of this section.

\begin{thm}\label{MainThm}
	Let $\pi$ be an irreducible representation of $K$ with be a highest weight $\mu$. Choose a positive root system $\Delta^+(\fg,\ft)$ making $\mu+2\rho_c$ dominant.  Let $C(\fg)$ be the closed dominant Weyl chamber corresponding to $\Delta^+(\fg, \ft)$. Then the inequality \eqref{the-ineq} is strict
	if and only if one of the following conditions holds:
	\begin{enumerate}
		\item[\emph{(a)}] $\mu+2\rho_c$ is irregular for $\Delta(\fg, \ft)$.
		\item[\emph{(b)}] $\mu-w\rho+\rho_c\notin C(\fk)$ for all $w\in W(\fg, \ft)$ such that $\mu+2\rho_c\in w C(\fg)$.
	\end{enumerate}
\end{thm}

\begin{proof}
Let $P(\cdot)$ be the projection map to the cone $C(\fg)$. It suffices to show that
\begin{equation}\label{ineqs}
\norm{\pi}_{\rm lambda}=\norm {P(\mu+2\rho_c-\rho)} \leq \norm{\mu+2\rho_c-\rho}\leq \norm{\pi}_{\rm spin},
\end{equation}
that the first equality happens if and only if (a) holds, and that the second equality happens if and only if (b) holds.

By the Pythagorean theorem, the first inequality in \eqref{ineqs} holds, and it becomes an equality if and only if
$P(\mu+2\rho_c-\rho)=\mu+2\rho_c-\rho$, which is equivalent to $\mu+2\rho_c-\rho\in C(\fg)$.
 The latter is equivalent to (a) since $\mu+2\rho_c\in C(\fg)$ and $\mu+2\rho_c$ is integral.

Now let us consider the second inequality in \eqref{ineqs}. We collect all $w\in W(\fg, \ft)$ such that $\mu+2\rho_c\in wC(\fg)$ as $W^1(\mu)$.  Since $\mu+2\rho_c\in C(\fk)$, it follows that $W^1(\mu)\subseteq W^1$. Moreover, the identity element $e\in W^1(\mu)$ due to $\mu+2\rho_c\in C(\fg)$.


Using Lemma \ref{lemmaD} and \ref{DominantNorm}, we have that
	\begin{equation}\label{RoughSpinIneq}
		\normspin{\pi}=\min_{w\in W^1}
		\norm{	
			\{\mu-w\rho+\rho_c\}+\rho_c
		}\geqslant
		\min_{w\in W^1}\norm{\mu-w\rho+2\rho_c}.
	\end{equation}
Take $\lambda=\mu+2\rho_c$ and $\Delta^+=\Delta^+(\fg,\ft)$ in Lemma \ref{LammaForLambdaAndRho}. We have
\begin{equation}\label{ineq-W2}
	\norm{\mu-w\rho+2\rho_c} \geq \norm{\mu-\rho+2\rho_c}.
\end{equation}
Furthermore, the inequality \eqref{ineq-W2} is  strict when $w \notin W^1(\mu)$; yet it is an equality when $w\in W^1(\mu)$. Now the second inequality in \eqref{ineqs} follows from (\ref{RoughSpinIneq}) and \eqref{ineq-W2}.

Assume (b) holds. For any $w\in W^1\setminus W^1(\mu)$, one has that
$$
\norm{\{\mu-w\rho+\rho_c\}+\rho_c}
\geqslant\norm{\mu-w\rho+2\rho_c}
>\norm{\mu-\rho+2\rho_c}.
$$
On the other hand, for all $w\in W^1(\mu)$, one has that
$$
\norm{\{\mu-w\rho+\rho_c\}+\rho_c}
>\norm{\mu-w\rho+2\rho_c}
=\norm{\mu-\rho+2\rho_c}.
$$
The first strict inequality is due to the assumption that $\mu-w\rho+\rho_c\notin C(\fk)$ and Lemma \ref{DominantNorm}.

Assume (b) does not hold. Then there exists some $w_0\in W^1(\mu)$ such that $\mu-w_0\rho+\rho_c\in C(\fk)$.
Therefore,
$$
\norm{\{\mu-w_0\rho+\rho_c\}+\rho_c}
=\norm{\mu-w_0\rho+2\rho_c}
=\norm{\mu-\rho+2\rho_c}.
$$
Since we have proven that
$$
\displaystyle\min_{w\in W^1}
\norm{	
	\{\mu-w\rho+\rho_c\}+\rho_c}\geq \norm{\mu-\rho+2\rho_c},
$$
we must have $\normspin{\pi}=\norm{\mu-\rho+2\rho_c}.$
	
To sum up, the two inequalities  in \eqref{ineqs} are controlled by (a) and (b), respectively. Thus $\normspin{\pi}>\normlambda{\pi}$ happens if and only if at least one of (a) and (b) holds.
\end{proof}

We record an interesting corollary from the above proof.

\begin{cor}
	If $\mu+\rho_c-w\rho\in C(\fk)$ for some $w\in W^1(\mu)$. Then
	\begin{equation}
		\normspin{\mu}=\norm{\mu+2\rho_c-\rho}.
	\end{equation}
\end{cor}

\section{Application to tempered Dirac series}

We call an irreducible tempered representations with non-zero Dirac cohomology \emph{a tempered Dirac series}.
Combining Theorems \ref{MainThm} and \ref{baseThm}, we have the following.

\begin{thm}\label{thmApplication}
	Let $X$ be a tempered $(\fg,K)$-module with real infinitesimal character. Let $\pi$ be the unique lowest $K$-type of $X$ which has a highest weight $\mu$. Then $H_D(X)=0$ if and only if
\begin{itemize}
\item[\emph{(a)}] $\mu+2\rho_c$ is irregular for $\Delta(\fg, \ft)$; or
\item[\emph{(b)}] $\mu-w\rho+\rho_c$ is not dominant for $\Delta^+(\fk, \ft)$ for any $w\in W^1(\mu)$.
\end{itemize}
\end{thm}

\begin{example}\label{exam-W1=e}
In the special case that $W^1=\{e\}$, which is met for complex Lie groups, $SL(2n+1, \bbR)$, $SL(n, \bbH)$ and the linear $E_{6(-26)}$, we always have that $\mu+2\rho_c$ is regular for $\Delta(\fg, \ft)$.  Thus $H_D(X)=0$ if and only if $\mu-\rho+\rho_c$ is dominant for $\Delta^+(\fk, \ft)$.
\end{example}

When $\#W^1>1$, pick up two distinct elements $w_1, w_2$ from $W^1$ such that $w_1 C(\fg) \cap w_2 C(\fg)$ is a codimension one facet of $w_1 C(\fg)$. Then condition (a) holds for any $\mu$ such that $\mu+2\rho_c\in w_1 C(\fg) \cap w_2 C(\fg)$. This suggests that the tempered Dirac series of $G$ should be divided into $\#W^1$ parts by those irreducible tempered $X$ such that $H_D(X)$ vanishes.

From now on, we shall use a circle to stand for a $K$-type, and paint it if and only if \eqref{the-ineq} is an equality. Let us see some concrete examples.

\begin{example}\label{exam-SL2R}
Consider ${\rm SL}(2, \bbR)$, where $\Delta(\fg, \ft)=\Delta(\fs, \ft)=\{ \pm 2\}$. Then $\#W^1=2$ and $C(\fg)\cap s C(\fg)=\{0\}$, where $s$ is the non-trivial element in $W^1$. Condition (b) does not take effect here since $\Delta(\fk, \ft)$ is empty. Thus $\mu=0$ is the unique $K$-type such that $\normspin{\mu}>\normlambda{\mu}$, and the tempered Dirac series of $SL(2, \bbR)$  are separated into two parts. See Figure \ref{fig-sl2r}.
\end{example}

\begin{figure}
	\begin{tikzpicture}
\foreach \x in {-4,-3,-2,-1,1,2,3,4}
	\fill (\x,0) circle (0.1);
	
\draw (0,0) circle (0.1);

\node[below] at (-4,0){-4};
\node[below] at (0,0){0};
\node[below] at (4,0){4};
	\end{tikzpicture}
\caption{Some $K$-types of $SL(2, \re)$}\label{fig-sl2r}
\end{figure}
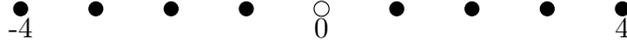

\begin{figure}
	\begin{tikzpicture}
\foreach \x in {2,3,4}
	\fill (\x,\x) circle (0.1);

\fill (4,3) circle (0.1);
\fill (3,2) circle (0.1);
\fill (4,2) circle (0.1);

\foreach \x in {1,2,3,4}
	\draw (\x,1) circle (0.1);

\foreach \x in {2,3,4}
\fill (\x,0) circle (0.1);

\foreach \x in {0,1}
\draw (\x,0) circle (0.1);

\foreach \x in {1,2,3}
	\draw (\x,-\x) circle (0.1);

\foreach \x in {-1,0}
	\draw (\x,-1) circle (0.1);
	
\foreach \x in {2,3,4}
	\fill (\x,-1) circle (0.1);

\foreach \x in {-2,0,1,3,4}
	\fill (\x,-2) circle (0.1);

\foreach \x in {-3,-2,0,1,2,4}
	\fill (\x,-3) circle (0.1);

\foreach \y in {-2,-3}
	\draw (-1,\y) circle (0.1);

\foreach \x in {-4,-3,-2,0,1,2,3}
	\fill (\x,-4) circle (0.1);

\foreach \x in {-1,4}
	\draw (\x,-4) circle (0.1);

\node[below] at (0,0) {(0,0)};
\node[below] at (4,4) {(4,4)};
\node[below] at (-4,-4) {(-4,-4)};
\node[below] at (4,-4) {(4,-4)};

	\end{tikzpicture}    
	\caption{Some $K$-types of $Sp(4, \re)$}\label{fig-sp4R}
\end{figure}
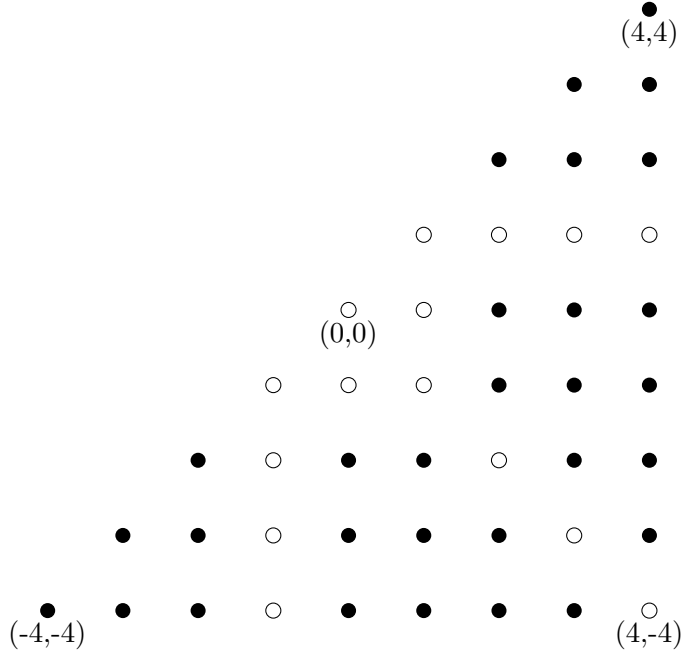

\begin{example}\rm
Consider $G=Sp(4,\re)$. Let $K=U(2)$ and $T=U(1)\times U(1)$. Thus $\fk$ has a one-dimensional center. Fix
$$
\Delta^+(\fk,\ft)=\{(1,-1)\},\quad \Delta^+(\fg,\ft)=\{(1,-1),(2,0),(0,2),(1,1)\}.
$$
The corresponding simple roots are $\alpha_1=(1,-1)=2\rho_c$, and $\alpha_2=(0,2)$.  The highest weight of a $K$-type is represented by a pair of integers $(x, y)$ such that $x\geq y$.

Condition (a) of says that the $K$-types on the three lines $y=1$, $x=-1$ and $y=-x$ should \emph{not} be painted. These lines intersect at the point $(-1,1)$, which is  $-2\rho_c$. Condition (b) further says that $(1,0)$ and $(0,-1)$ should \emph{not} be painted.  Now Figure \ref{fig-sp4R} suggests that the tempered Dirac series of $Sp(4, \re)$ are separated into four parts.
\end{example}

\begin{example}\label{exam-G22}
Let $G$ be $G_{2(2)}$, the linear split $G_2$, which is centerless, connected, but not simply connected. We adopt the simple roots of $\Delta^+(\fg,\ft)$ and $\Delta^+(\fk,\ft)$ as in Knapp \cite{Kn}. Let $\alpha_1$ be the short simple root and $\alpha_2$ be the long one. In this case, $\Delta(\fg,\ft)$ is of type ${\rm G_2}$, while $\Delta(\fk,\ft)$ is of type ${\rm A_1\times A_1}$. We fix $\Delta^+(\fk,\ft)=\{\gamma_1, \gamma_2\}$, where $\gamma_1:=\alpha_1$ and $\gamma_2:=3\alpha_1+2\alpha_2$. Let $\omega_1,\omega_2$ be the fundamental weights for $\Delta(\fk,\ft)$ such that
$(\omega_i,\alpha_j^\vee)=\delta_{ij}.$ The $K$-types are parameterized via the highest weight theorem by $[a,b]:=a\omega_1+b\omega_2,~a,b\in \mathbb Z_{\geqslant 0}$ such that $a+b$ is even.

We show some of the $K$-types in Figure \ref{FigureG22}, where the $a$-coordinates of the bottom line are $0, 2, 4, 6, 8, \dots$, and so are the $b$-coordinates of the left-most column.

Now condition (a) says that $K$-types on the two lines $a=b$ and $a=3b+4$ should \emph{not} be painted.  These two lines intersect at $[-2,-2]=-2\rho_c$. From Figure \ref{FigureG22}, one sees that the tempered Dirac series are divided into three parts by the two lines. Condition (b)  further says that $[2,0]$ should \emph{not} be painted.

To sum up, we have recovered Corollary 8.4 of \cite{DDY}.
\end{example}

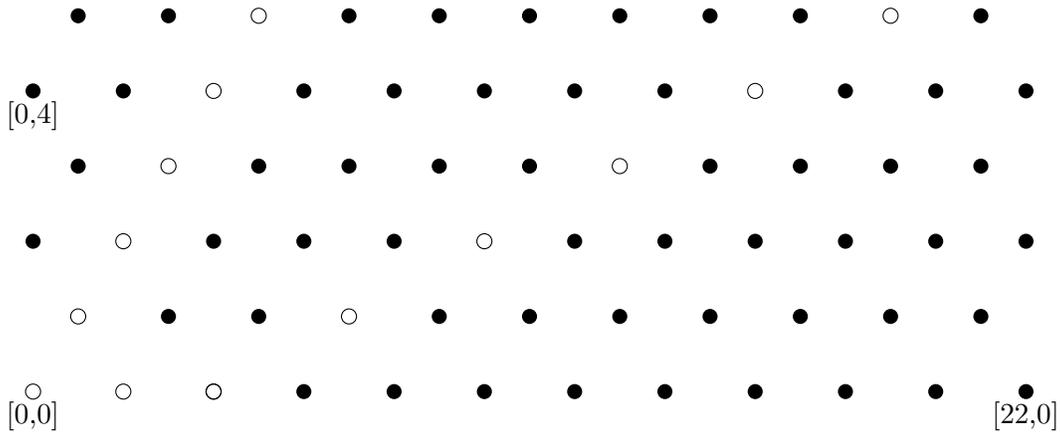
\begin{figure}
	\begin{tikzpicture}
		\foreach \x in {6,8,...,22}
		\fill (\x*0.6,0) circle (0.1);
		\foreach \x in {3,5,9,11,...,21}
		\fill (\x*0.6,1) circle (0.1);
		\foreach \x in {0,4,6,8,12,14,...,22}
		\fill (\x*0.6,2) circle (0.1);
		\foreach \x in {1,5,7,9,11,15,17,...,21}
		\fill (\x*0.6,3) circle (0.1);
		%
		\foreach \x in {0,2,6,8,...,14,18,20,22}
		\fill (\x*0.6,4) circle (0.1);
		\foreach \x in {1,3,7,9,...,17,21}
		\fill (\x*0.6,5) circle (0.1);
		\foreach \x in {0,2,4}
		\draw (\x*0.6,0) circle (0.1);
		\foreach \y in {1,2,3,...,5}
		\draw (\y*0.6,\y) circle (0.1);
		\foreach \y in {1,2,3,...,6}
		\draw (\y*1.8+0.6,\y-1) circle (0.1);
		
		\node[below] at (0,0){[0,0]};
		\node[below] at (0,4){[0,4]};
		\node[below] at (13.2,0){[22,0]};
	\end{tikzpicture}
\caption{Some $K$-types of the linear split $G_2$}\label{FigureG22}
\end{figure}

\centerline{\scshape Funding}
Dong is supported by the National Natural Science Foundation of China (grant 12171344).


\begin{thebibliography}{1}
	
\bibitem{Ca} J.~Carmona, \emph{Sur la classification des modules admissibles irr\'eductibles}, pp.11--34 in Noncommutative Harmonic Analysis and Lie Groups, J. Carmona and M. Vergne, eds., Lecture
Notes in Mathematics \textbf{1020}, Springer-Verlag, New York, 1983.

\bibitem{DD} J.~Ding and C.-P.~Dong, \emph{Spin Norm, K-Types, and Tempered Representations}, J. Lie Theory \textbf{26} (2016), 651--658.	


\bibitem{DDY} J.~Ding, C.-P.~Dong,  and L.~Yang, \emph{Dirac series for some real exceptional Lie groups}, J.  Algebra \textbf{559} (2020) 379--407.	
	

	
\bibitem{D13} C.-P.~Dong, \emph{On the Dirac cohomology of complex Lie group representations}, Transform. Groups \textbf{18} (2013), 61-79. [Erratum: Transform. Groups \textbf{18} (2013), 595--597.]	

	

\bibitem{HC} Harish-Chandra, \emph{Harmonic analysis on real reductive Lie groups. I. The theory of the constant term} J. Funct. Anal. \textbf{19} (1975), 104--204.

\bibitem{HP} J.-S. Huang and P.~Pand\v zi\'c, \emph{Dirac
	cohomology, unitary representations and a proof of a conjecture of
	Vogan}, J. Amer. Math. Soc.  \textbf{15} (2002), 185--202.
	
\bibitem{Hum} J.~E.~Humphreys, \emph{Introduction to Lie Algebras and Representation Theory},
	Springer-Verlag, New York, 1972.

\bibitem{Kn} A.~Knapp, \emph{Lie Groups, Beyond an Introduction, 2nd edition}, Birkh\"auser, 2002.

\bibitem{SRV} S.~Salamanca-Riba and D.~Vogan,  \emph{On the classification of unitary representations of reductive Lie groups}, Ann. of Math. \textbf{148} (1998), 1067--1133.
	
\bibitem{T} P.~Trapa, \emph{A parametrization of $\hat K$ (after Vogan)}, Notes from an AIM workshop, July 2004.
	
\bibitem{V81} D.~Vogan,  \emph{Representations of Real Reductive Groups}, Birkh\"{a}user, 1981.
	
\bibitem{V84} D.~Vogan, \emph{Unitarizability of certain series of representations}, Ann. of Math. \textbf{120} (1984), 141--187.	
	

\end{thebibliography}
\end{document}